\documentclass[11pt,fullpage]{article}
\usepackage{amsmath}
\usepackage{amsfonts,amssymb}
\usepackage{amsthm}
\usepackage{hyperref}
\usepackage{graphics,graphicx}
\usepackage{subfig}
\usepackage{float}
\usepackage{color}

\usepackage[right = 2.5cm, left=2.5cm, top = 2.5cm, bottom =2.5cm]{geometry}
\usepackage{authblk}

\newcommand{\Real}{\mathbb R}
\newcommand{\norm}[1]{\left\|#1\right\|}
\newcommand{\abs}[1]{\left\vert#1\right\vert}
\newcommand{\set}[1]{\left\{#1\right\}}

\newcommand{\G}{G_\alpha}
\newcommand{\K}{\mathcal{K}}
\newcommand{\J}{\mathcal{J}}
\newcommand{\U}{\mathcal{U}}
\newcommand{\T}{\mathcal{T}}
\newcommand{\stu}{u^*}

\newcommand{\F}{\mathcal{F}}
\newcommand{\N}{\mathcal{N}}

\newcommand{\be}{\begin{equation}}
\newcommand{\ee}{\end{equation}}

\theoremstyle{break}
\newtheorem{theorem}{Theorem}
\theoremstyle{remark}
\newtheorem{remark}{Remark}

\theoremstyle{theorem}
\newtheorem{prop}{Proposition}

\theoremstyle{definition}

\theoremstyle{lemma}
\newtheorem{lemma}{Lemma}

\begin{document}

\title{Analysis of a model for the dynamics of riots}

\author[a]{H. Berestycki}
\author[b]{N. Rodr\'iguez\footnote{corresponding author: nrodriguez@unc.edu}}
\affil[a]{\footnotesize{EHESS, CAMS, 190 - 198 avenue de France, 75013 Paris, France.}}
\affil[b]{\footnotesize{UNC Chapel Hill, Department of Mathematics,
Phillips Hall, CB$\#$3250, Chapel Hill, NC 27599-3250 }}

\maketitle 

\begin{abstract}
This paper is concerned with modeling the dynamics of social outbursts of activity, such as protests or rioting activity.
In this sequel to our work in \cite{Berestycki2014}, written in collaboration with J-P. Nadal, we model the effect of restriction of information and 
explore the effects that it has on the existence of {\it upheaval waves}.  The systems involve the coupling
of an explicit variable representing the intensity of rioting activity and an underlying (implicit)
field of social tension.  We prove the existence of global solutions to the Cauchy problem in $\Real^d$
as well as the existence of {\it traveling wave solutions} under certain parameter regimes.  We furthermore explore the effects of heterogeneities
in the environment with the help of numerical simulations, which leads to pulsating waves in certain cases.  
We analyze the effects of periodic domains as well as the {\it barrier} problem with the help of numerical simulations and discuss many open problems.

\end{abstract}

\section{Introduction}
The need to understand the spatio-temporal dynamics of social outbursts, such as protests or rioting activity,
has been highlighted with many current events.  For example, the fatal shooting of Michael Brown in Ferguson, Missouri (USA) 
by a white police officer in August of 2014 led to a two-week period of protests,  which subdued until the grand jury's
decision not to indict the officer in question \cite{NYT14}.  A similar wave of protests was later observed in New York City (USA) after the
grand jury's decision not to indite the police officer involved in Eric Garner's chokehold case \cite{Goodman2014}.  As many individuals throughout
the U.S. and the rest of the world are keeping a close watch on these and similar cases there is a risk of increased social tension spreading throughout the country 
and abroad \cite{Eversley2014, Rodriguez2014b, Hickey2014}.  We think of protests as outbursts of social activity
which are induced by an initial shock and grow, due to a self-reinforcement mechanism and/or continuing external shocks, for some period of time. 
We sometimes refer to this initial shock as the triggering event. 
Clearly, it is important not only to understand when these outbursts of activity will occur, but also how they will spread geographically.
This motivated us to introduce various models that include what we believe to be the bare essential mechanisms which are necessary 
to capture the stylized facts of outbursts of social activity, in particular related to protests or civil unrest, in \cite{Berestycki2014} written in collaboration with 
 J-P. Nadal. 
The models we introduced involve the coupling of an
explicit variable representing the intensity of rioting activity and an underlying (implicit) field
of social tension. The effects of exogenous and endogenous factors as well as
various propagation mechanisms are also included in the system.  It must be said that  
the purpose of these series of works is not to understand the economic, social or political origins of riots, 
and much less to discuss the legitimacy or lack thereof of 
any riot.  These issues are much too complex to be studied with mathematical models which are obtained by making simplifying assumptions.   
Instead, our objectives are of an exploratory nature, where we wish to capture some qualitative facts of riots using our model with the aim of gaining
insight into how riots spread and what the impact of restriction of information and various communications networks is on the spread and duration of rioting activity. 

The models of \cite{Berestycki2014} introduce new and interesting mathematical structures for which fundamental questions, such as 
whether the systems have global solutions, have yet to be studied.  The existence of traveling wave solutions in the case of non-local diffusion is also 
an open question, which we hope to address in future work.  The purpose of this manuscript is four-fold.  First, we introduce a model that captures
the effect that restriction of information has on the success or failure of a riot.  Second, we explore the existence of traveling wave solutions (upheaval waves),
which provides insight into how fast riots spread.  Third, we explore the effects of various heterogeneous environments on the success or failure of riots through
numerical experiments.  
Finally, we take the first step in developing the mathematical theory called for by the family of models introduced in \cite{Berestycki2014}, 
by proving the existence and uniqueness of solutions to the Cauchy problem.  Furthermore, we discuss many open problems that are of interest from both the point of view of
mathematics and the application.   The approach of this work bears similarities to a recent literature on the use of mathematics in the
analysis of uncivil and criminal activities \cite{Short, Berestycki2010,Schweitzer2000}, showing for instance that patterns, which are useful to understand, 
emerge when looking at the macroscopic scale. 

As the family of models we introduced in our work with Nadal  \cite{Berestycki2014} capture the effects of geographic proximity as well as social connections between different 
regions that are prone to protests or rioting activity, they are natural candidates to shed light onto answer to questions such as why some
protest or even revolutions skip certain countries.  There is evidence that restriction of information hinders revolutions and
that increased social media access and technology has the opposite effect \cite{Ghannam2011,Ang1996,Attia2011,Tufekci2012}.  
Naturally, increased access to technology has decreased the ability of any state to isolate their populations from the rest of the world.
The models we introduce here modify those from \cite{Berestycki2014} in order
to take into account the effect of restriction of information in a system.   
Under some parameter regime our model is able to capture ``upheaval waves," a disturbance of rioting activity or
protests that propagates through space.  These disturbances are sparked by a triggering event of sufficient intensity.
It is of interest to quantify the speed at which these upheaval waves move, something that is facilitated by the models we introduce.  
Moreover, as the restriction of information is reduced we observe a shift in the qualitative behavior of the upheaval waves.   Particularly,
when there is some restriction of information there exists upheaval waves, under certain parameter regimes, which travel at a unique speed.  However, 
as the restriction of information is lifted completely
the waves can travel with speed $c$ for $c\ge c^*$. Presumably, the intensity and range of the initial
shock will have an effect on the spreed.  In fact, from numerical simulations we observe that initial levels of rioting activity of the form $u(x,0) = O(e^{-kx})$ lead to traveling
wave solutions whose speed depends on $k$.  This is analogous to what is observed in the well-known Fisher-KPP equation, see for example \cite{Sherratt1998} and references within. 

There have been other models of rioting activity introduced in the literature, see for example \cite{Braha2012, Lang2014, Davies2013}.  
We refer the interested reader to \cite{Berestycki2014}
for a summary of these works.  One of our contribution is the introduction of a continuous model, which is amenable toward more rigorous mathematical analysis.
In fact, our model takes a new direction aiming at using systems of PDEs, an approach which has been useful in  
gaining insight into other social phenomena, 
such as urban crime \cite{Short2008, Berestycki2010, Berestycki2013c},
social segregation \cite{Rodriguez2014a}, opinion dynamics \cite{Schweitzer2000}.

{\it Outline of the paper:}  In section \ref{sec:model} we introduce the basic model along with some variations which we study in this work.  In
section \ref{sec:analysis} we prove some preliminary results needed for future sections.  The
existence of global solutions to the Cauchy problem in $\Real^d$ for $d\ge 1$ is discussed in section \ref{sec:cauchy}.  In section \ref{sec:tw}
we discuss the existence of traveling waves solutions and conclude with a discussion in section \ref{sec:disc}.

\section{Dynamics of social outbursts of rioting activity} \label{sec:model}

\subsection{The basic model}
In this section we summarize the models of \cite{Berestycki2014}.  The derivation began with
a network of nodes which represent centers which are prone to rioting activity.  We approximate the
global behavior of such network by deriving the appropriate continuum limit, which then enables us to perform rigorous mathematical 
analysis.  Let $\Omega\subset \Real^2$ represent the domain of interest, {\it i.e.} a city or country. 
For a location $x\in \Omega$ and time $t$ we let $u(x,t)$ represent the {\it level of rioting activity} and $v(x,t)$ represent the {\it social tension} in the system.  
The fundamental assumption in our model is a dichotomy of external influences and internal influences
in rioting activity: an external shock is necessary to initiate the process (exogenous factor), but the outburst is maintained by 
a self-reinforcement mechanism (endogenous factors).  As these outburst are short-lived one must also include mechanisms that lead to an eventual self-relaxation of the 
activity.  For simplicity we summarize the assumptions that we make below and refer the reader to \cite{Berestycki2014} for a more detailed derivation. 
\begin{enumerate}
\item {\it Exogenous factors}:  there are external factors that can either trigger or rekindle an outburst of protests or rioting activity.  An example of this 
would be the shooting of Michael Brown in Ferguson, Missouri (USA) or the acquittance of the police officers involved in the beating of Rodney King
in 2001 in Los Angeles, California (USA).  Such events are modeled as point sources that are placed at the time and location of the triggering events.  
For example, if $n$ external events or shocks occur at places and times $(s_i,t_i)\in \Omega\times \Real$ then the source term for the social tension is given by:   
\[s(x,t) :=A_i\sum_{i=1}^n \delta_{t=t_i,s=s_i},\]
where $A_i$ is the intensity of the $i^{th}$ event.  For the sake of simplicity, we will always assume that $A_i=A_0$ for all $i=1...n$ in what follows. 

\item {\it Endogenous factors}:  once an outburst of activity is initiated, it is maintained by a self-reinforcement mechanism, up to a certain carrying capacity.  
This is analogous to the repeat and near-repeat victimization effect assumed in the models of Short {\it et al.} in 
\cite{Short2009} and Berestycki and Nadal in \cite{Berestycki2010}.  In essence, this is the phenomena that criminal activities in certain location lead 
to an increased probability of that same location falling victim to more criminal activities. 
The self-reinforcement mechanism is modeled
by the function $G(z),$ which is included as a source term in the dynamics of the level of rioting activity and satisfies: 
\[
G(z)> 0\;\text{for}\;z\in(0, z_0), \;G(0)= 0\;\text{and}\; G(z)\leq 0\;\text{for} \;z\geq z_0.
\]
 An example of this is a KPP-type term: $G(z) = z (1 - \frac{z}{z_0})$ for $z\in (0, z_0)$ for some $z_0>0$, where $z_0$ is the carrying capacity.  The
need for a carrying capacity arises from the fact that there are a finite number of people that will protest or in the case of riots a finite number of cars and buildings
that can be destroyed. 

\item {\it Critical threshold for the social tension}: 
an outburst of rioting activity or protests only occurs if the social tension in the system is sufficiently high.  We introduce
 a {\it critical threshold} parameter $a\in \Real^+$ which determines when the endogenous factors begin to play a role.  In other words, when the
social tension is  
above this critical value the self-reinforcement mechanism is initiated.  This gives rise to a transition function that depends on the social tension $r(u).$
This transition function can be, for example, a sigmoid function: 
\[
r(z)= \frac{\gamma}{1 + e^{-\beta(z-a)}},\] 
where $\beta>0$ provides a measure of the transition between a system without self-reinforcement and a self-reinforcing system and $\gamma$ gives us the units of the inverse of time.  In other words,
it provides a measure of how fast the transition is between a system that does not include the endogenous factors and 
a system with the full-force of these factors.  Note that large $\beta$ provides fast transitions.  The coefficient $a$ is the 
critical tension threshold, which provides the level of tension necessary for the self-reinforcement mechanism to begin.  
   
\item {\it Self-relaxation}:  there is a natural decay rate for the level of rioting activity, denoted by the parameter $\omega.$  
Moreover, the level of rioting activity presumably will have an effect on the social tension and so it is 
natural to assume that the social tension deceases at a rate which is proportional to the level of rioting activity.  Specifically, the decay is given by $-h(u)v$ where
\[
h(u) = \frac{\theta}{(1+mu)^p,}
\] 
for some $\theta, p,m>0$.  The parameter $\theta$ measures the decay in the absence of any riots, $p$ controls the effect that the level of rioting activity
has on the social tension, and $m$ provides the units of the inverse of the level of rioting activity.  
Note that when $u$ is large $h(u)$ is small, modeling the fact that high levels of rioting activity will lead to a slow decay of the social tension level. 

\item {\it Geographic proximity}:  while geographic proximity does not provide the full picture, especially nowadays, it still play an important role.  For example, if we assume that 
the level of rioting activity $u(x,t)$ is simply given by the fraction of people protesting then local diffusion is appropriate in modeling the spread of the activity.  Geographic proximity
can also have an effect on the spread of social tension.  Indeed, one might feel the effects of what occurs in neighboring locations more significantly. 
This effect is modeled by the Laplace operator and we will denote by $D_1,D_2\ge 0$ the diffusivity coefficients of $u$ and $v$ respectively.    

\item {\it Social proximity}:  one can argue the social connection between two places
might be even more important than geographic proximity.  Two locations that are socially connected might be more likely to influence each other than two places that
are neighbors.  The importance of this consideration is seen, for example, in the fact that 
urban centers are more likely to influence smaller cities or towns than vise-versa.  Thus, it is important to include these
social connections.  These are modeled though the interaction kernel, $\J(x,y)$, which provides a measure of
how much influence location $y$ has on location $x$.  In a sense, $\J(x,\cdot)$ gives the {\it domain of influence} of $x$, that is, the strength of the influence that
other locations have on $x$. On the other hand, $\J(\cdot, y)$ gives the 
{\it range of influence} of location $y$, that is, the influence that location $y$ has on all other locations.

\end{enumerate}
Combining the above assumptions we state the most general model:
\begin{subequations}\label{sys:cont1}
\begin{align}
&u_t(x,t)= D_1\Delta u(x,t) +r (v(x,t)) G(u(x,t)) - \omega u(x,t),\label{eq:u}\\ 
&v_t(x,t)= D_2\Delta v(x,t) +\kappa\int_\Omega\hspace{-4pt} \J(x,y)v(y)\;dy-  (h(u(x,t)) - \eta) v(x,t)+s(x,t)+v_b,\label{eq:v}
\end{align}
\end{subequations}
for $x\in \Omega$ and $t>0$, where $\Omega \subset \Real^d$ or $\Omega = \Real^d.$   
In \eqref{eq:v} $v_b$ is the base social tension value.  Note that we model the level of rioting activity that exceeds the base rate. 
This will be accompanied with the initial conditions: 
\begin{align}\label{cond:ic}
u(x,0)=u_0(x)\quad\text{and}\quad v(x,0)=v_0(x),
\end{align}
where $u_0(x),v_0(x)\in L^\infty(\Omega)\cap L^1(\Omega)$ and no-flux boundary conditions if $\Omega$ is bounded and 
limiting conditions if $\Omega =\Real^d$.  That is, from now on if $\Omega$ is bounded we assume that
\begin{align}\label{cond:bc}
\frac{\partial}{\partial \vec n}u(x,t)=\frac{\partial}{\partial \vec n}v(x,t) =0\quad\text{for}\quad x\in \partial \Omega\quad\text{and}\quad t>0,
\end{align}
 where $\vec n$ is is the outward unit normal.
Some of the theory, such as the existence of solutions to the Cauchy problem, 
will cover the cases when $d\ge 1.$  The theory of traveling wave solutions will be restricted to $d=1$ since our main interest
is the study of planar traveling waves, which can be reduced to the understanding of the evolution of the wave profile in one dimension.


\begin{remark}
Generally speaking the dispersal of social tension in system one \eqref{sys:cont1} will be either strictly local or strictly non-local, that is $D_2\ge 0$ and $\kappa =0$ or
$D_2=0$ and $\kappa \ge 0.$  However, we state the system in its generality to provide a general global well-posendess theory for the Cauchy problem. 
\end{remark}

\subsection{Modeling restriction of information}
In order to understand the effect that restriction of information has on the ignition and spread of 
protests we modify system \eqref{sys:cont1} to include the fact that a
revolution will begin only if the number of protests are above a certain threshold.  The paper \cite{Lang2014} used this concept with the aim of finding a 
model for the dynamics of revolutions, such as the Arab Spring.  
The idea is that small protests will simply go unnoticed. Moreover, it is also reasonable to assume that the tighter the 
restriction of information the larger the protests have to be in order for them to gain any momentum.  We choose to model the restriction of 
information with the reaction term
$\G(u)$ which is assumed to be of ignition-type: 
\begin{align}\label{eq:igni}
 \G(z) = 0\;\text{for}\;z\in [0,\alpha], \G(0) >0\; \text{for}\; z\in\left(\alpha,z_0\right)\;\text{and}\;\G(z)\le 0 \;\text{otherwise},
\end{align}
for $\alpha \in [0,z_0].$ We introduce $\alpha$ to represent the level of restriction of information in a system with larger values of $\alpha$ corresponding to higher levels of restriction of information.  
Naturally, there are many questions of interest which can be explored, such as what the effect of technology reaching the hands of more and more people is on the existence 
and speed of upheaval waves.  Mathematically, this corresponds to determining the effect of taking the limit $\alpha\rightarrow 0$. Another question that can be explored
is that of why revolutionary waves skip certain countries.  Needless to say, our work can only shed some light on this phenomenon and we do not claim to provide definite
answers.

\subsection{Non-local diffusion}
When the social connection between two locations is only a function of the distance between those locations then the interaction kernel has the form $\J(x,y) = \J(x-y)$ and the system
will have the following form:
\begin{subequations}\label{sys:cont2}
\begin{align}
&u_t(x,t)= D_1\Delta u(x,t) +r (v(x,t)) \G(v(x,t)) - \omega u(x,t),\\
&v_t(x,t)= D_2\Delta v(x,t) +\kappa\int_\Omega \J(x-y)v(y)\;dy -  (h(u(x,t)) - \eta) v(x,t)+s(x,t)+v_b.
\end{align}
\end{subequations}
On the one hand, this restricts the types of social interactions that we consider significantly, on the other hand, system \eqref{sys:cont2}
is a useful simpler model that allows for the exploration of non-local influences.  For example, one can determine the role that the range of
influence plays on the initiation, duration, and spread of rioting activity.  
For this special class of kernels, it is reasonable to 
expect that system \eqref{sys:cont2} will exhibit traveling wave solutions, whereas there is no reason to expect such solutions
for a more general class of kernels.

\section{Preliminary analysis}\label{sec:analysis}
\subsection{Non-dimensional system} 
Before analyzing \eqref{sys:cont1} let us first write the dimensionless model in order to remove superfluous parameters.  
Now, as discussed in \cite{Berestycki2014} it is natural to make the assumption that the timescale over which the endogenous effects are 
observed is smaller than the exogenous effect, which then implies that $\omega>\theta$.  Under this assumption the 
natural time scale is $1/\omega$. The characteristic length scale is $\sqrt{D_1/\omega},$ which is roughly the distance over which
the level of rioting activity diffuses during the characteristic timescale.  A natural change of variables then is: 
\[
\tilde t = \omega t, \;\tilde x=\sqrt{\frac{\omega}{D_1}}x,\;\tilde u = \frac{u}{z_0},\;\text{and}\;\tilde v = \frac{\omega}{v_b}v.
\]
This gives the non-dimensional system: 
\begin{subequations}\label{sys:non_dim}
\begin{align}
&u_t = \Delta u + r (v)G_{\bar \alpha}(u) - u,\\
&v_t = D\Delta v +k \int \J(x,y)v(y)\;dy - (h(u)-k_2)v +1+s(x,t),
\end{align}
\end{subequations}
where the transition, decay, and source functions are now given by:
\[
r(z) = \frac{\rho}{1+e^{-\beta_1(z-\bar a)}},\quad h(z) = \frac{1}{(1+ \bar mz)^p},\quad\text{and}\quad s(x,t)=\tilde A\sum^n_{i=1}\delta_{t=t_i,x=\bar x_i},
\]
the non-dimensional parameters are:
\[
\rho = \frac{\gamma}{\omega},\;k = \frac{\kappa v_b}{\omega},\; k_2 =\frac{\eta}{\omega},\tilde A= \frac{A_0}{v_b}, \;\bar a = \frac{a\omega}{v_0},\;\beta_1=\frac{\beta v_b}{\omega}, 
D=\frac{D_2}{D_1},\;\bar m = mz_0,\;
 \bar\alpha = \alpha/z_0.
\]
and $G_{\bar\alpha}(u) = 0$ for $u\in [0,\bar\alpha]$ and $G_{\bar\alpha}(u) = (u-\bar\alpha)(1-u)$ for $u\notin [0,\bar\alpha]$ with $\bar \alpha \in [0,1]$.  
For notational simplicity we replace $\bar \alpha$ with $\alpha$ for the remainder of the paper.  Note also that since $G_\alpha$ was normalized we gained a
dimensionless parameter $\rho$ that measures the intensity of the self-reinforcement mechanism.  It will soon become apparent that this parameter, which we refer to as the
{\it self-reinforcement} parameter because it measures the strength of the endogenous factors,  
plays a significant role in the behavior of our model. 

\subsection{Constant steady-state classification}\label{sub:ss}
In this section we classify the constant steady-state solutions of \eqref{sys:non_dim} in the case of local diffusion, {\it i.e.} $k=0$.  We rewrite the system for convenience as follows:
\begin{subequations}\label{sys:non_dim_1}
\begin{align}
&u_t = \Delta u + \Phi(u,v),\\
&v_t = D\Delta v +k \int \J(x,y)v(y)\;dy +\Psi(u,v)+s(x,t),
\end{align}
\end{subequations}
where, 
\[
\Phi(u,v):= r (v) \G(u) - u \quad\text{and}\quad \Psi(u,v):=- (h(u)-k_2)v +1. 
\]
Any constant steady-state solution of \eqref{sys:non_dim_1} with $k = 0$ must satisfy:
\begin{align}\label{eq:alss}
v^*(u) = \frac{1}{h(u)-k_2}\quad\text{and}\quad \Phi(v^*(u),u)=0.  
\end{align}
Note that $v^*(u)$ is an increasing function of $u$ such that
\[
\lim_{u\rightarrow \bar u} v^*(u) = \infty, 
\]
for $\bar u = \frac{1}{\bar m}\left[\left(\frac{1}{k_2}\right)^{1/p} -1\right].$  Note that $\bar u>0$ as $k_2<1$ and that for any constant steady-state solution $(u_c,v_c)$ it holds that $u_c<\min\set{1,\bar u}$ and $v^*(0) \leq v_c<v^*(1)$.  
In order to be able to analyze the effect of the transition parameter $\beta$ we normalize the critical social tension value $a$ to be the average between $v^*(0)$ and $v^*(1)$.  
Then the number of constant steady-states will depend on the 
restriction of information parameter $\alpha$ and the interplay between the transition parameter $\beta$ and the self-reinforcement parameter $\rho$.  
We describe the different regimes
of the parameters below and illustrate them in Figures \ref{fig:bif1} and \ref{fig:bif2}.  The various characterizations of the regions in the figure were determined through a combination of analytical results and heuristics, and 
confirmed numerically.  See the Appendix (section \ref{sec:apx}) for more details.  It is important to keep in mind that this analysis is for the case when $r(u),h(u)$ and $G_\alpha(u)$ have the specific forms discussed above. 

\begin{enumerate}
\item {\it No restriction of information}: for $\alpha = 0$ there can be either one, two, or three constant steady-state solutions.   As can be observed in Figure \ref{fig:bif1} 
there exist two curves, $\beta_1(\rho)$ and $\beta_2(\rho)$, that separate the number of steady-state solutions.  A third curve $\beta_3(\rho)$ separates the system from being a 
monostable system, where there is only one stable constant steady-state solution, and a bistable system, where there are two stable constant steady-state solutions.  Finally, the curve
$\beta_4(\rho)$ divides the region where there are three constant steady-state solutions into two regions.  In one region the non-excited steady-state solution is globally stable
and in the other region the fully-excited steady-state solution is globally stable. We provide a more detailed summary of all of the regions below.   
\begin{itemize}
\item Figure \ref{fig:bif1} Region I: This region corresponds to a weak self-reinforcement effect.  That is, if $\rho$ is small then the only constant steady-state is the non-excited state $(0,v^*(0))$.  All solutions of the system \eqref{sys:non_dim_1} approach $(0,v^*(0))$ in the long-term, see Lemma \ref{lem:no_rev}.  

\item Figure \ref{fig:bif1} Region II:  This region corresponds to intermediate and strong reinforcement effect but slow to intermediate transitions.
As $\rho$ increases the existence of two or three constant steady-state solutions depends on the transition parameter $\beta$.  
Slow transitions lead to the existence of two constant steady-state solutions.  For $\beta$
below $\beta_3(\rho)$ the system can be classified as monostable with one stable steady-state solution and
the other unstable (this is region IIb of Figure \ref{fig:bif1}).  Transitions in an intermediate region (region IIa of Figure \ref{fig:bif1}) 
lead to a bistable system with two stable constant steady-state solutions. 
\item Figure \ref{fig:bif1} Region III:  This region corresponds to an intermediate to strong self-reinforcement effect and fast transition rates.  For
$\rho$ sufficiently large system \eqref{sys:non_dim_1} has three constant steady-state solutions.  These correspond to
the non-excited state $(0,v^*(0))$, the {\it warm state}, by which we mean a steady-state with a non-trivial amount of social tension and rioting activity but not the maximum level, and 
the excited state $(u^*,v^*(u^*))$.  The warm state is 
unstable in this regime and the non-excited and excited states are stable.    
For intermediate values of $\rho$ (region IIIa) the non-excited state is globally stable.  On the other hand, for $\rho$ to the right of
$\beta_4(\rho)$ the fully-excited state is globally stable.

\end{itemize}
\item {\it Moderate restriction of information}: for $0<\alpha<1$ there are generally either one or three constant 
steady-state solutions.   The existence of two steady-state 
solutions does occur when the parameters lie on the curve $\beta_1(\rho),$ which is illustrated of Figure \ref{fig:bif2}.  
However, in this cases the second non-zero steady-state solution is degenerate as $\Phi(u,v^*(u),)\le 0$ for all $u\ge 0$. 
The curve $\beta_2(\rho)$ separates the regions where $F(u)\le 0$ for all $0<u<\min\set{1,\bar u}$ and $F(u)>0$ for some
$u\in (0, \min\set{1,\bar u})$.  

\begin{itemize}
\item Figure \ref{fig:bif2} Region I: This region corresponds to weak self-reinforcement effects.  This if, if $\rho$ is small then the only constant steady-state solution is the non-excited state $(0,v^*(0))$. 
  As in the previous case of weak reinforcement mechanisms we expect that all solutions of system \eqref{sys:non_dim_1} approach $(0,v^*(0))$ in the long-term, see Lemma \ref{lem:no_rev}.

\item Figure \ref{fig:bif2} Region IIa: This region corresponds to medium-range reinforcement effects.  As $\rho$ increases there is a bifurcation into a regime where
three constant steady-state solutions exist.  In this case the non-excited state and the fully-excited state are both stable; however, the
non-excited state will be globally stable.  
\item Figure \ref{fig:bif2} Region IIb:  This region corresponds to large reinforcement effects.  As $\rho$ increases further there is a bifurcation into a regime where
the fully-excited state is now globally stable.  In this case the 
non-excited state is the preferred state. 
\end{itemize}
\item {\it Maximum level of restriction of information}: when $\alpha = 1$ then $(0,v^*(0))$ is the only constant steady-state solution. 
\end{enumerate}


\begin{figure}[H]
  \center
\includegraphics[width=1.05\textwidth]{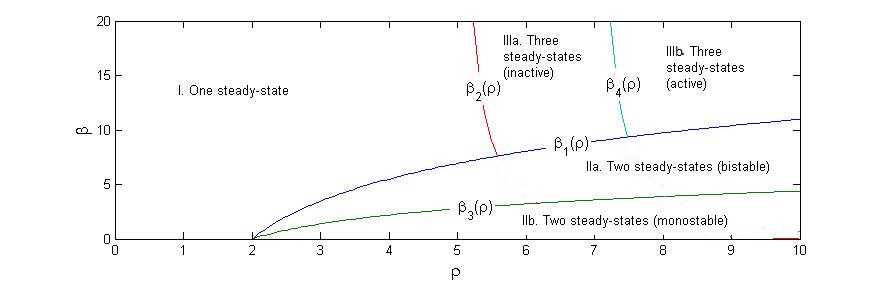}
 \caption{\small{Bifurcation diagram for $\rho$ and $\beta$ of \eqref{sys:non_dim} when $k=0$ and there 
is no restriction of information.}\label{fig:bif1}}
\end{figure}

\begin{figure}[H]
  \center
\includegraphics[width=1.05\textwidth]{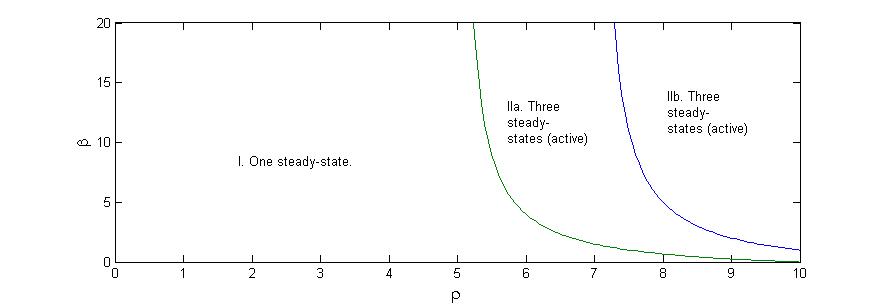}
 \caption{\small{Bifurcation diagram for $\rho$ and $\beta$ of \eqref{sys:non_dim} when $k=0$ and there 
$0<\alpha<1$}.\label{fig:bif2}}
\end{figure}

\subsection{Comparison principle}
Many of the results in this work will rely on the construction of suitable super-solutions and sub-solutions.  In order to be able to use these
solutions as barriers we need the following simple comparison principle.  
\begin{lemma}[Comparison principle]
Let $(u_1,v_1)$ and $(u_2,v_2)$ be two solutions of \eqref{sys:non_dim_1} with initial data satisfying $u_1(x,0)\ge u_2(x,0)$
and $v_1(x,0)\ge v_2(x,0)$ in $\Real^d$ for $d\ge 1$.  Then $u_1(x,t)\ge u_2(x,t)$
and $v_1(x,t)\ge v_2(x,t)$ for all $t\ge 0$.  
\end{lemma}

\begin{proof}
Let $u = u_1-u_2$ and $v=v_1-v_2,$ then $(u,v)$ satisfy:
\begin{subequations}\label{sys:comp_prin}
\begin{align}
&u_t = \Delta u + \left(r (v_1)\G'(\bar u)-1\right) u + r'(\bar v)\G(u_2) v,\\
&v_t = D\Delta v +k \int \J(x,y)v(y)\;dy +(k_2 - h(u_1))v - h'(\tilde u)v_2 u,
\end{align}
\end{subequations}
for $\bar u,\tilde u,\bar v\ge 0$.  Since $r'(\bar v)\G(u_2)\ge 0$ and $h'(\tilde u)v_2\le 0$ the system is monotone and thus by standard theory
we conclude that $u(x,t),v(x,t) \ge 0$ for all $t\ge 0$.  
\end{proof}

\subsection{Liouville type results}
As the ultimate distributions of the level of rioting activity and the social tension in the system are important to understand, the question of whether the 
only steady-state solutions to \eqref{sys:non_dim_1} are the constant solutions is not only of mathematical interest but also of general interest.  It is certainly the case that
for parameters in regions I of Figure \ref{fig:bif1} and Figure \ref{fig:bif2} where the reaction terms are strictly negative, the only steady-state solutions will be
$(0,v^*(0))$.  Moreover, in the case when the social tension does not diffuse we can prove that the only steady-state solutions are indeed the constant ones.  

\begin{prop}\label{prop:LV}
Let $D=k = \alpha = 0$ and $\rho, \beta$ be chosen from region IIa or IIb of Figure \ref{fig:bif1}.  Let $(u,v)$ be a classical non-negative bounded steady-state
solutions to \eqref{sys:non_dim_1} with no-flux boundary conditions \eqref{cond:bc} on a bounded domain 
then either $(u,v)\equiv (0,v^*(0))$ or $(u,v)\equiv(u^*,v^*(u^*)).$
\end{prop}

\begin{proof}
Under the conditions stated in the proposition any steady-state solution will satisfy:
\begin{align*}\left\{\begin{array}{l}
\Delta u + r \left(\frac{1}{h(u)-k_2}\right)\G(u) - u = 0,\\
v^*(u) = \frac{1}{h(u)-k_2}.\end{array}\right.
\end{align*}
Define $f(u) = r \left(\frac{1}{h(u)-k_2}\right)\G(u) - u$ for which, based on the chosen parameters, there exists a unique $u^*$ such that:
\[
f(0)=f(u^*)=0,\;f(u)>0\;\text{for}\;u\in (0,u^*)\;\text{and}\;f(u)<0\;\text{otherwise}.
\]
Without loss of generality assume let us assume that $u\not\equiv 0.$  By the strong maximum principle
we conclude that $u(x)>0$ for all $x\in \bar\Omega$.  Let $m=\min_{x\in\bar\Omega} u(x)>0$ and  $M=\max_{x\in\bar\Omega} u(x)>0$.  
Since $f(u)$ is positive between $(0,u^*)$ then
by the strong maximum principle and Hopf lemma we know that $m\ge u^*.$   On the other hand, $f(u)<0$ for $u>u^*$ so we can conclude, in a
similar fashion, that $M\leq u^*$, which implies that $u\equiv u^*$ and thus $v=v(u^*)$. 
\end{proof}

It would be interesting to find the parameter regimes, if such exists, where the only steady-state solutions are the constant ones for the fully-parabolic system.  
For example, when the parameters lie in the region IIa and IIb of Figure \ref{fig:bif1}, so that system \eqref{sys:non_dim_1} with $k =0$ is monostable, are the only steady-state
solutions the zero steady-state, $(0,v^*(0)),$ and the excited steady-state, $(u^*,v^*(u^*))$.  This would be a natural extension from results that are known to hold
for the single parabolic Fisher-KPP equation.

\section{The Cauchy problem}\label{sec:cauchy}
This section is devoted to the study of the Cauchy problem in $\Real^d$ for $d\ge 1$ with one triggering event. We will take the triggering event to be part of the initial condition.  
More specifically, we consider the system:
\begin{subequations}\label{sys:gen}
\begin{align}
&u_t = \Delta u + r (v)\G(u) - u,\quad\text{for}\quad x\in\Real^d,\; t>0,\label{eq:u1}\\
&v_t = D\Delta v +k \int \J(x,y)v(y)\;dy - (h(u)-k_2)v +1,\quad\text{for}\quad x\in\Real^d,\; t>0,\label{eq:v1}\\
&u(x,0)=u_0(x)\quad\text{and}\quad v(x,0) =v_0(x) +A\delta_{x=\bar x}, \quad\text{for}\quad x\in\Real^d,\; t=0.
\end{align}
\end{subequations}

Let $C_b(\Real^d)$ denote the space of continuous bounded functions on $\Real^d$ and 
consider the Banach space $\U= C_b(\Real^d)\times C_b(\Real^d)$ endowed with
with the $L^\infty-$norm.  We obtain the following global existence result. 
  
\begin{theorem}[Global solutions to the Cauchy problem]\label{thm:cauchy}
Let $(u_0(x),v_0(x)) \in \U,$ $D\ge 0,$ and 
\begin{align}\label{cond:kernel}
\sup_{x\in\Real^d}\int_{\Real^d}\J(x,y)\;dy<\infty.
\end{align} 
Then, there exists a unique, positive, and global solution in time
$(u(x,t),v(x,t))\in \U$ to \eqref{sys:gen}. 
\end{theorem}

\begin{proof}
The proof is done in two steps.  In the first step we prove the local existence of a solution in a suitable subspace of $\U$. 
The second step is a simple continuation argument.  \\
\\
{\it Step 1:} (Local existence) 
The solution to \eqref{sys:gen} with $D>0$ can be re-written as:
\begin{subequations}\label{sys:newform}
\begin{align}
&u(x,t) = e^{-t} \K(x,t)\ast u_0(x) + \int_0^t e^{-t+s} \K(x,t-s)\ast r (v(x,s)) G_\alpha(u(x,s))\;ds,\label{def:u} \\
&v(x,t)= \K_{D}(x,t)\ast v_0(x) +\int_0^t \K_D(x,t-s)\ast \left(k\int_\Real \J(x,y) v(y,s)\;dy -(h(u(x,s))-k_2)v(x,s)+1\right)\;ds. \label{def:v}
\end{align}
\end{subequations}
where $\K$ is the heat kernel with diffusivity constant one and $\K_D$
is the heat kernel with diffusivity constant $D$.  If $D=0$ then \eqref{def:v} instead has the form:
\begin{align}\label{def:d_0}
&v(x,t)= v_0(x) +\int_0^t \left(k\int_\Real \J(x,y) v(y,s)\;dy -(h(u(x,s))-k_2)v(x,s)+1\right)\;ds. 
\end{align}

Let $\T[(u,v)]:=(\T_1(u,v),\T_2(u,v))$ be a map on $\U$ defined by:  
\begin{align}
\T_1(u,v) := &\;e^{-t} \K(x,t)\ast u_0(x) + \int_0^t e^{-t+s} \K(x,t-s)\ast r (v(x,s)) G_\alpha(u(x,s))\;ds,\\
\T_2(u,v) := &\; \K_{D}(x,t)\ast v_0(x)+\int_0^t \K_D(x,t-s)\;ds \notag \\
&+\int_0^t \K_D(x,t-s)\ast \left(k \int \J(x,y) v(y,s)\;dy-(h(u(x,s))-k_2)v(x,s)\right)\;ds. 
\end{align}
If $D=0$ then we instead set: 
\begin{align}
\T_2(u,v) :=  v_0(x) +\int_0^t \left(k\int_\Real \J(x,y) v(y,s)\;dy -(h(u(x,s)-k_2)v(x,s)+1\right)\;ds. 
\end{align}

Our goal is to find a subset of $\U$ that is invariant under the transformation $\T$. For this purpose, given $T,R>0$ define the set: 
\[
X_{R,T} = \set{(u,v) \in \U: \norm{(u,v)}_{X_T}\leq R},
\]
where, 
\[
\norm{(u,v)}_{X_T} =  \norm{u(t)- e^{-t} \K(x,t)\ast u_0(x)}_\infty+\norm{v(t)-v_0(x) -t}_\infty,
\]
if $D=0.$ On the other hand, if $D>0$ then we instead use the norm: 
\[\norm{(u,v)}_{X_T} \hspace{-4pt}=  \norm{u(t)- e^{-t} \K(x,t)\ast u_0(x)}_\infty\hspace{-3pt}+\norm{v(t)-\K_{D}(x,t)\ast v_0(x)-\hspace{-3pt}\int_0^t \hspace{-4pt}\K_D(x,t-s)\;ds}_\infty\hspace{-8pt}.\]

Take $(u,v)\in X_{R,T},$ to show invariance we obtain the following estimate with the use of Young's inequality for convolutions:
\begin{align}\label{est:1}
\norm{\T_1(u,v)- e^{-t} \K(x,t)\ast u_0(x)}_\infty &\leq \int_0^t e^{-t+s}\norm{\K(x,t-s)\ast r (v(x,s)) G_\alpha(u(x,s))}_\infty\;ds \notag\\
& \leq \int_0^t e^{-t+s}\norm{\K(x,t-s)}_1\norm{r (v(x,s)) G_\alpha(u(x,s))}_\infty\;ds\notag \\
& \leq C\int_0^t e^{-t+s}\norm{r (v(x,s)) G_\alpha(u(x,s))}_\infty\;ds\notag \\
& \leq C R (1-e^{-t}), 
\end{align}
as $\norm{r (v(x,s)) G_\alpha(u(x,s))}_\infty \leq \abs{G'_\alpha}\norm{u}_\infty\leq CR$.  Similarly, for $\T_2$ for $D\ge 0$ we obtain the estimate:
\begin{align}\label{est:2}
\norm{\T_2(u,v)- \K_{D}(x,t)\ast v_0(x)-\hspace{-3pt}\int_0^t \hspace{-3pt}\K_D(x,t-s)\;ds}_\infty 
\vspace{-12pt}&\leq C\int_0^t  \norm{\int_{\Real^d} \J(x,y) v(y,t) dy}_\infty\vspace{-32pt}ds\notag \\
&\quad-\int_0^t(h(u(x,s))-k_2)v(x,s)\;ds\notag\\
& \leq C\left(\sup_{x\in\Real^d}\int_{\Real^d}\J(x,y)\;dy \right)\norm{v}_\infty t\notag\\
& \leq CRt,
\end{align}
where we have used condition \eqref{cond:kernel} and the fact that for $\norm{u}_\infty$ sufficiently small then $h(u)-k_2>0.$
From \eqref{est:1} and \eqref{est:2} we observe that for $T$ sufficiently small
$T[u,v]\in X_{T,R}$. 
To show that $\T$ is a strictly contractive map, take $(u_1,v_1),(u_2,v_2)\in X_{T,R},$ we obtain 
\begin{align*}
&\norm{\T_1(u_1,v_1)-\T_1(u_2,v_2)}_\infty\leq \abs{G_\alpha'}\sup_{0\le t\le T}\norm{u_1-u_2}_\infty(1-e^{-t}),\\
&\norm{\T_2(u_1,v_1)-\T_2(u_2,v_2)}_\infty\leq \left(\sup_{x\in \Real^d}\int_{\Real^d} \J(x,y)\;dy+C\right) \sup_{0\le t\le T}\norm{v_1-v_2}_\infty t+C\hspace{-5pt}\sup_{0\le t\le T}\norm{u_1-u_2}_\infty t.
\end{align*}
Therefore, for $T$ sufficiently small there exists and $0<k <1$ such that: 
\[\norm{\T[u_1,v_1]-\T[u_2,v_2]}_\U\le k \norm{(u_1,v_1)-(u_2,v_2)}_\U.\]  Given that $\T$ is
a strict contraction then there exists unique $(u(x,t), v(x,t))\in \U,$ which is a fixed point of $\T$ and thus satisfies \eqref{sys:gen}.

{\it Step 2:} (Continuation of the solution)  Recall that we are looking for positive solutions.  A simple continuation argument allows us to conclude that either $T=\infty$ or
\begin{align}\label{cond:finite}
\lim_{t\rightarrow T^{-}} u(x,t) = \infty\quad \text{or}\quad \lim_{t\rightarrow T^{-}} v(x,t) = \infty.
\end{align}
Since, $u_t<0$ for any $u>u^*$ we see that $\norm{u(x,t)}_\infty\leq \max\set{\norm{u_0(x)}_\infty, u^*}$.  On the other hand,
$v(x,t)\leq v_0(x)e^{Mt}$ where 
\[
M:=\sup_{x\in \Real^d} \int J(x,y)\;dy. 
\]  
This rules out \eqref{cond:finite} and thus the solutions exists for all time.   
\end{proof}

%

\section{Traveling wave solutions without non-local effects}\label{sec:tw}
From the point of view of the modeling it is of particular interest to look at the invasion of protests into regions that are experiencing low levels of protesting or rioting activity.  
A benefit of modeling riots with system \eqref{sys:non_dim} is that a wave of rioting activity or protests that spread geographically
is represented by {\it traveling wave solutions}, {\it i.e.} solutions that propagate in a given direction with a constant profile.  
In the context of the model we refer to these solutions as {\it upheaval waves} or {\it protests waves}.
In fact, for the local model ($k =0$) we can determine what parameter regimes lead to such
solutions for system \eqref{sys:gen} with $k =0$.  To be more precise, an {\it upheaval wave} is given by $u(x,t)=\phi(x-ct)$ and
$v=\psi(x-ct)$ with $c\in \Real$ which satisfy:
\begin{align}\label{sys:traveling_wave}
\left\{\begin{array}{l}
\phi''(z)+ c\phi'(z) +\Phi(\phi(z),\psi(z))=0, \\
\psi''(z)+c\psi'(z) + \Psi(\phi(z),\psi(z)) = 0,\\
0\leq \phi(z)\leq u^*,\; v^*(0)\leq \psi(z)\leq v^*(u^*),\\
\phi(-\infty) = u^*,\;\phi(+\infty) = 0,\psi(-\infty) = v^*(u^*),\;\psi(+\infty) = v^*(0),
\end{array}\right.
\end{align}
for $z=x-ct$.  In system \eqref{sys:traveling_wave}, $(u^*,v^*(u^*))$ represents the excited constant steady-state solution and $(0,v^*(0))$ represents the non-excited steady-state solution, 
with no rioting activity and the lowest social tension level.  
It is useful to define the following function:  
\[
g_\alpha(u) := r\left(\frac{1}{h(u)-k_2}\right) G_\alpha (u),
\]
as the speed of the upheaval waves will depend on the quantity:
\begin{align}\label{def:sign}
F(u):=\int_0^{u} g_\alpha(s)\;ds. 
\end{align}
We first state a result about the eventual extinction of rioting activity when the parameters
are chosen from regions I of Figure \ref{fig:bif1} and Figure \ref{fig:bif2}.  Of course, this implies the non-existence of 
upheaval waves.  
\begin{lemma}[Extinction of rioting activity] \label{lem:no_rev}
Let $k=0,$ $\alpha =1$ or $0\le \alpha$ and $\rho,\beta$ chosen from regions $I$ of Figures \ref{fig:bif1}-\ref{fig:bif2}.  
 For any initial data $(u_0(x),v_0(x))\in L^1(\Real^d)\times L^1(\Real^d)$ and source term:
\[
s(x,t) = A\sum_{i=1}^N \delta_{t=t_i,x=x_i},
\] 
the solution $(u(x,t),v(x,t))$ to \eqref{sys:non_dim_1} approaches $(0,v^*(0)
)$ as $t\rightarrow \infty$.  Specifically,
\[
u(x,t)\rightarrow 0\;\text{and}\;v(x,t)\rightarrow v^*(0),
\]
as $t\rightarrow \infty,$ uniformly in $x.$
\end{lemma}	
Lemma \ref{lem:no_rev} has the simple interpretation that in a fully censored state upheaval waves will never exist.
At the same time even as the restriction of information decreases, if $\rho$ is not sufficiently large there is also 
no possibility of an upheaval wave taking off.  How large $\rho$ has to be depends on the restriction of information; the 
higher the restriction of information the higher $\rho$ needs to be in order for a riot to get momentum.    

\begin{proof}(Lemma \ref{lem:no_rev})
First assume that there is only one triggering event, in which case we assume that the source term can be taken as 
part of the initial condition and $s(x,t)=0.$
Let $\beta$ and $\rho$ satisfy the hypothesis of the lemma, there exists an $\tau>0,$ which is sufficiently small, such that
$r(v)G_\alpha(u)-u\le -\tau u.$  For example, when $\alpha = 1$ then $\tau = 1$.  Let 
\begin{align}\label{def:u}
\bar u(x,t)=e^{-\tau t}u_0(x)\ast \K(x,t),
\end{align}
where $\K$ is the heat kernel.  Note that \eqref{def:u} satisfies $\bar u_t = \Delta \bar u-\tau \bar u.$  Thus, any solution $u(x,t)$ of \eqref{sys:non_dim_1} is bounded above by $\bar u(x,t)$
for all $t>0$. 
Moreover, we have the bound
\[
\norm{u(x,t)}_\infty \le e^{-\tau t}\norm{u_0(x)\ast \K(x,t)}_\infty\le C(\norm{u_0(x)}_\infty,\norm{\K}_1) e^{-\tau t}. 
\]  
In particular, $\norm{u}_\infty\rightarrow 0 $ as $t\rightarrow \infty$ uniformly in $x$. 
Furthermore, note that $h(u)\ge h(\norm{u}_\infty)$ since $h(u)$ is a decreasing function of $u$.  Let $w = v-v^*(0) = v-\frac{1}{h^*(0)-k_2},$
then $w$ satisfies:
\[
w_t =D\Delta w -(h(u)-k_2)w+\left(1-\frac{h(u)-k_2}{h(0)-k_2}\right)+s(x,t). 
\]
Using \eqref{def:u} we obtain that 
\[
\lim_{t\rightarrow \infty}\norm{1-\frac{h(u)-k_2}{h(0)-k_2}}_\infty = 0,
\]
also uniformly in $x$ since $u$ converges uniformly.  Thus, we observe that $\norm{w(x,t)}\rightarrow 0$ as $t\rightarrow \infty$ uniformly in $x.$ 
The result for the case when there are $N$ triggering events follows by the argument above by taking the initial conditions to be the $u_0(x)=u(x,t_N)$ and $v_0(x)=v(x,t_N)$. 
\end{proof}

On the other end of a the spectrum, as the restriction of information begins to be lifted there are certain parameter
regimes that will allow the existence of revolutionary waves.  This is expressed in the following result.

\begin{theorem}[Upheaval waves and restriction of information]\label{thm:traveling_wave} 
Let $0\leq \alpha <1.$
\begin{enumerate}
\item For $\rho$ and $\beta$ in regions IIa, IIIa, or IIIb of Figure \ref{fig:bif1} or regions $IIa$ and $IIb$ of Figure \ref{fig:bif2},
the system \eqref{sys:traveling_wave} admits a unique solution $(\phi(z),\psi(z),c^*)$ up to translations. 
Furthermore, 
\begin{itemize}
\item[(i)] $sgn(c^*)=sgn(F(\stu)),$ where $F(u)$ is defined by \eqref{def:sign}. 
\item[(ii)] $\phi'(z)< 0$ and $\psi'(z)<0$ for all $\abs{z}<\infty$.  
\end{itemize} 
\item If $\alpha=0$ and $\rho$ and $\beta$ lie in region IIb of Figure \ref{fig:bif2} then \eqref{sys:traveling_wave} admits a unique $c^*\in \Real^+$ and
solutions $(\phi_c(z),\psi_c(z))$ with $z=x-ct$ for any $c\geq c^*$. 
On the other hand, when $c<c^*$, such solutions do not exist.
\end{enumerate}
\end{theorem}

In contrast to Lemma \ref{lem:no_rev}, Theorem \ref{thm:traveling_wave} tells us that as the
restriction of information decreases in a state the possibility of the existence of upheaval waves begins to increase.  In fact, 
when $\alpha = 0$ there exists a wide parameter regime in the $\rho-\beta$ plane that allows for
the existence of traveling waves.  Even more interesting is the possibility of upheaval waves
that move with an arbitrarily large speed (depending on the initial shock).    For $0<\alpha<1$
if the growth rate $\rho$ is sufficiently large there exists unique upheaval waves
whose speed depends on \eqref{def:sign}.  On the other hand, if $\rho$ is not sufficiently large 
there is the possibility of retreating waves, where the non-excited state $(0,v^*(0))$ 
invades regions with a high level of rioting activity $(u^*,v^*(u^*))$. 

Figure \ref{fig:tw_exp} illustrates various
numerical experiments of system \eqref{sys:non_dim_1}.  Figures \ref{fig:tw1} and \ref{fig:tw2} illustrate two simulations with the exact
parameters ($\rho$ and $\beta$ in region IIa of Figure \ref{fig:bif1}) but with different initial conditions.  
In particular, the initial conditions used in the simulation illustrated by Figure \ref{fig:tw1}
is $O(e^{-x}),$ which affect a larger region, leads to a faster upheaval wave than that illustrated in Figure \ref{fig:tw2}.  In fact, in Figure \ref{fig:tw2}
the initial condition is $O(e^{-3x}).$  Figure \ref{fig:tw3} illustrates another simulation with no restriction of information, but with parameters in the 
bistable region.  In this case all initial conditions lead to waves of the same speed. 
Figures \ref{fig:tw4}-\ref{fig:tw6} illustrate simulations with some restriction of information.  Figure \ref{fig:tw4} illustrates a bistable 
upheaval wave resulting from low restriction of information where the excited state invades the non-excited state.  As the restriction of information is increased
the traveling wave slows down and eventually we obtains a stationary wave,  as is illustrated in Figure \ref{fig:tw5}, and even a retreating wave when the
restriction of information is sufficiently high (Figure \ref{fig:tw5}).  In the latter case one observes that the non-excited state 
invades the excited state.

\begin{figure}[H] 
  \center
  \subfloat[$\alpha =0$: $u_0(0) = 5e^{-x}$ ]{\label{fig:tw1}\includegraphics[width=0.45\textwidth]{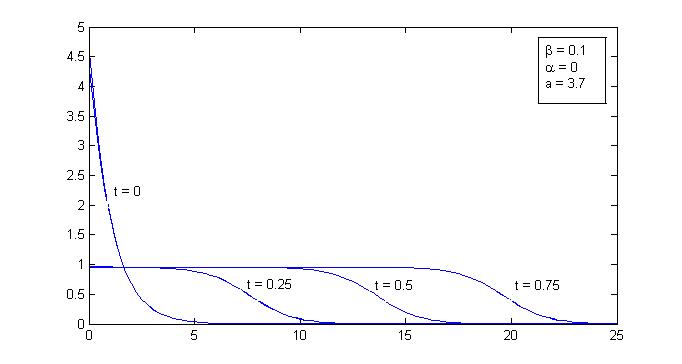}}\quad\quad\quad
 \subfloat[$\alpha =0$: $u_0(0) = 5e^{-3x}$]{\label{fig:tw2}\includegraphics[width=0.45\textwidth]{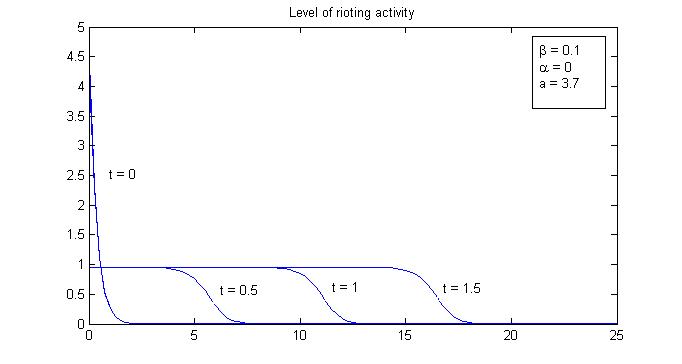}}\\
  \subfloat[$\alpha =0$, fast transition]{\label{fig:tw3}\includegraphics[width=0.45\textwidth]{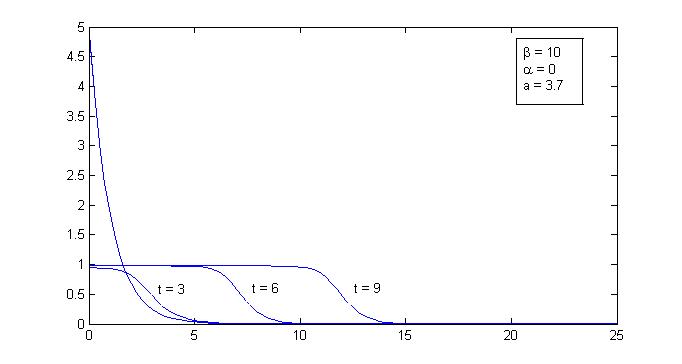}}\quad\quad
 \subfloat[$0<\alpha <1$, bistable wave]{\label{fig:tw4}\includegraphics[width=0.45\textwidth]{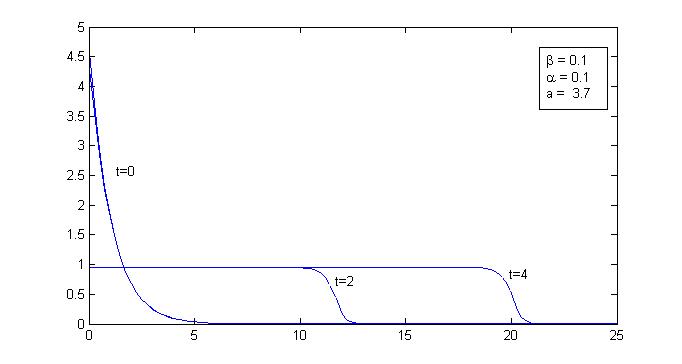}}\\
  \subfloat[$0<\alpha <1$, stationary wave]{\label{fig:tw5}\includegraphics[width=0.45\textwidth]{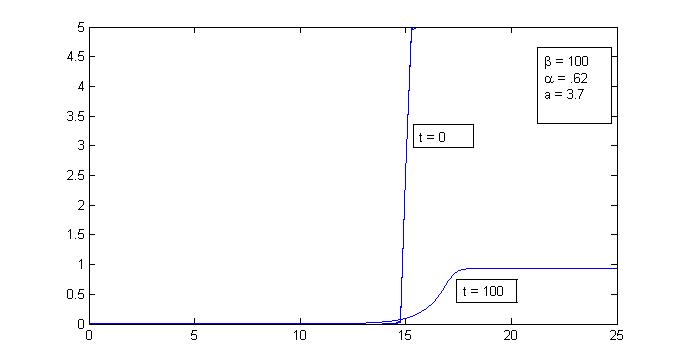}}\quad\quad
 \subfloat[$0<\alpha <1$, retreating wave]{\label{fig:tw6}\includegraphics[width=0.45\textwidth]{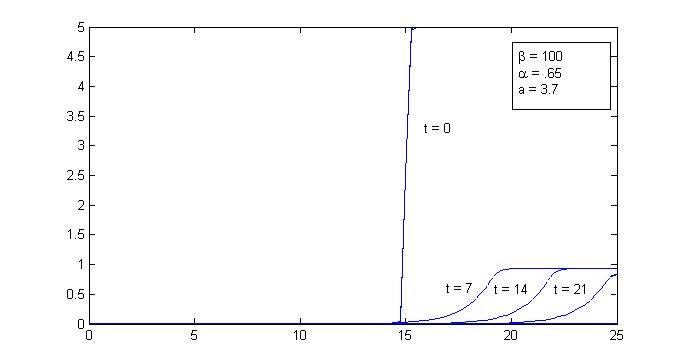}}\\
 \caption{\small{Traveling wave solutions for the level of rioting activity with various parameters and initial conditions.}\label{fig:tw_exp}}
\end{figure}

\begin{proof}(Theorem \ref{thm:traveling_wave})
We rewrite the system given by \eqref{eq:u1}-\eqref{eq:v1} in vector-matrix form:
\be\label{sys:mv}
\vec{u}=\vec F[\vec u],
\ee
where $\vec{u}=(u,v)$ and $\vec F[\vec u]=(\Phi(u,v),\Psi(u,v))$. First, note that $\Phi_v(u,v) = r'(v)G_\alpha(u)\ge 0$
for all $\alpha\in[0,1]$ since $r'(v)>0$.  Moreover,  $\Psi_u(u,v) = \frac{p}{(1+u)^{p-1}}v\ge 0$
for all $u,v\ge 0$ and so the system \eqref{sys:mv} is monotone.  Now, the stability of any constant steady-state solution $(u_c,v_c)$ to 
\eqref{sys:mv} is given by:
the sign of the eigenvalues of the matrix:
\begin{align}\label{mat}
\vec{F}'[(u_c,v_c)]=\left[\begin{array}{ll}
r(v_c)\G'(u_c)-1&r'(v_c)\G(u_c)\\
-h'(u_c)v & -(h(u)-k_2)\end{array}\right].
\end{align}
For notational simplicity we denote the trace of \eqref{mat} by:
\[\text{Tr}[(u_c,v_c)] = r(v_c)\G'(u_c)-1-(h(u)-k_2),\]
and the determinant by
\[\text{D}[(u_c,v_c)] = -(r(v_c)\G'(u_c)-1)(h(u)-k_2)+r'(v_c)\G(u_c)h(u).\]

{\it Case 1}: Let us first consider the case when the parameters are chosen from regions IIa, IIIa, or IIIb of Figure \ref{fig:bif1} or IIa-IIb of Figure \ref{fig:bif2}.
For the non-excited steady-state $(0,v^*(0))$ we have that
\[\text{Tr}[(0,v^*(0))] = \frac{\rho}{1+e^{-\beta(v^*(0)-a)}}\G'(0)-1-(1-k_2)< 0.\]
Indeed, when $\alpha =0$ we know that $\G'(0)=0$ and the inequality is clear.  For the case when $\alpha=0$ recall that $\rho$ and $\beta$ are chosen such that 
$  \frac{\rho}{1+e^{-\beta(v^*(0)-a)}}\G'(0)<2-k_2$.  Similarly, we observe that
\[\text{D}[(0,v^*(0))] = -(r(v^*(0))\G'(0)-1)(1-k_2)+r'(v^*(0))\G(0)>0.\]  Thus, the eigenvalues of
$\vec{F}'[(0,v^*(0))]$ lie on the left-half plane.  Now, for the excited state $(u^*,v^*(u^*))$ we have that
\[\text{Tr}[(u^*,v^*(u^*))] = r(v^*(u^*))\G'(u^*)-1-(1-k_2)<0,\]
since $\G'(u^*)<0$ and by the same token
\[\text{D}[(u^*,v^*(u^*))] = -(r(v^*(u^*))\G'(u^*)-1)(1-k_2)+r'(v^*(u^*))\G(u^*)>0.\]
Again, from this it is clear that the eigenvalues of $\vec{F}'[(u^*,v^*(u^*))]$ lie on the left-half plane.  For parameters in the region IIa on Figure \ref{fig:bif1} there are only two steady-states. 
In the remaining regions there is a warm state $(u_w,v^*(u_w))$ such that
\[\text{Tr}[(u_w,v^*(u_w))] = \frac{\rho}{1+e^{-\beta(v^*(u_w)-a)}}\G'(u_w)-1-(1-k_2)>0,\] since $\G'(u_w)$ is sufficiently large so that
$  \frac{\rho}{1+e^{-\beta(v^*(u_w)-a)}}\G'(u_w)>2-k_2$.  Thus, at least one of the eigenvalues of $\vec{F}'[(u_w,v^*(u_w))]$ lie on the right-half plane.

{\it Case 2:}  Let us now consider the case when the parameters are chosen from region IIb in Figure \ref{fig:bif2}.  In this case, there are two steady-state
solutions $(0,v^*(0))$ and $(u^*,v^*(u^*))$.  The analysis for the excited state is the same as above so that all of the eigenvalues of the Jacobian evaluated 
at this steady-state lie on the left-half plane.  However, for $(0,v^*(0))$ the parameters are such that $\frac{\rho}{1+e^{-\beta(v^*(0)-a)}}\G'(0)>2-k_2$ so that
\[\text{Tr}[(0,v^*(0))] >0.\]  This implies that at least one of the eigenvalues of the Jacobian evaluated at this steady state lies on the right-half plane.   
 
Thus, the existence and monotonicity of the traveling wave solutions $(\phi(z),\psi(z))$ that satisfy \eqref{sys:traveling_wave} now follow directly from 
Theorem 2.1 in \cite{Volpert1994} for {\it 1} and Theorem 2.2 for {\it 2}
and we are left to verify $(i)$.  For this purpose, 
let $U(z)$ be a traveling wave solution with speed $c'$ satisfying 
\begin{align}\label{eq:tw_sing}
\left\{\begin{array}{l}
U''+c'U'+g_\alpha(U) = 0,\\
0\leq U\leq u^*,\\
U(-\infty) = u^*,\;U(+\infty) = 0.
\end{array}\right.
\end{align}
Note that such a solution exists and that the speed $c'$ depends on $\F(u^*),$ both in sign and in magnitude.  We also know that
$U(z)$ is monotone decreasing.  
Now, let $V = v^*(U)$ so that 
\begin{align}\label{prop:v}
\lim_{z\rightarrow -\infty} V(z) = v^*(u^*),\;\lim_{z- \infty}V(z) = v(0),\;\text{and}\;\Psi(V,U) = 0.  
\end{align}

Without loss of generality assume that $c' >0$ and assume for contradiction that $c^* \le 0.$  Now, define $w(z,t)=u(z+ct, t),$
$b(z,t) = v(z+ct,t),$ and the operator
\begin{align}
\N\bf v: = \left\{\begin{array}{l}
w_t -w_{zz}-c^* w_z-\Phi(b,w),\\
b_t - b_{zz} -c^*b_z-\Psi(b,w).\end{array}\right.
\end{align}
For any ${\bf v} :=(w,b)$ which satisfies \eqref{sys:traveling_wave} it holds that $\N\bf v=0.$  Applying the operator $\N$ to 
${\bf v}_1 = (U,V)$ gives:
\begin{align}
\N\bf v: = \left\{\begin{array}{l}
-U''-c^*U'-g_\alpha(U), \\
- V'' -c^*V'.\end{array}\right.
\end{align}
First, note that since $U$ satisfies \eqref{eq:tw_sing} we obtain that $-U''-g_\alpha(U) <0$ and $-c^*U' \le 0.$ Moreover,
 using \eqref{prop:v} we have that $-V''-c^*V' \le 0$, which can be seen by integrating the left-hand-side and noting that  
\[
-c^*(V(\infty)-V(-\infty))\le 0.
\]
Thus, we conclude that $\N{\bf v}_1 \le 0$, so that $(U,V)$ is a super-solution. 
  To exclude any issues at $\pm\infty$ let us choose $z_0\in\Real$
such that 
\begin{align}\label{cond:less}
U(z+z_0) < \phi(z) \quad\text{and}\quad V(z+z_0)<\psi(z)\quad\text{for all}\;z\leq R, 
\end{align}
for some $R\in \Real$ and such that $U(z+z_0)<\delta$ for some $\delta>0$ and all $z>R$. 
  In fact, \eqref{cond:less} must hold for all $z\in \Real.$  To observe this note that
\begin{align*}
&\phi''(z) + g_\alpha(\phi) \le 0,\\
& U''(z) + g_\alpha(U) \ge 0.
\end{align*}
Thus, if we define $m(z) = U(z)-\phi(z)$ then it satisfies $m''(z)+q(x)m(z) \ge 0$ where
\[
q(x) = \frac{g_\alpha(U) - g_\alpha(\phi)}{U-\phi}.  
\] 
Now, for parameters $\rho$ and $\beta$ in regions III of Figure \ref{fig:bif1} or regions II on Figure \ref{fig:bif2}, since $U<\delta$ 
we know that $q(x)\le 0$ in the set $\set{m\ge 0}$.  An application of the strong
maximum principle to $m''(z)+q(x)m(z) \ge 0$ on the set $\set{m > 0}$ implies that $m\le 0$ as
$m=0$ on the boundary, which is a contradiction, and so we conclude that
$m \le 0$ for all $z\in \Real$.  A similar argument works for $\phi(z)$ and $V.$  Thus, we know that
\[
U(z)\le \phi(z)\quad\text{and}\;V(z)\le \psi(z),
\] 
for all $z$.  The assumption that $c^* \le 0$ leads to a contradiction, which implies that $c^*> 0$. 
Similarly, we can conclude that if $c' = 0$ then $c^* = 0$ and if $c'<0$ then $c^*<0$.    

\end{proof}

\section{Heterogeneous environments}\label{sec:het}
As for ecological models, it is clear that the heterogeneities in the environment play a significant role on the ignition, spread, and propagation of rioting activity.  
Indeed, laws, unemployment rates, forms of government, etc. vary drastically between countries and even from state to state within a country.
To take this even further, income inequality is observed between neighborhoods that are adjacent to each other.  It is therefore important 
to analyze how these heterogeneities impact the overall failure or success of 
an upheaval wave.  In an effort to gain some perspective in this direction we perform various numerical experiments in heterogeneous environments.
In general, these heterogeneities can be quite complex and difficult to analyze and a good first approximation is the so-called {\it patch model},
a terminology borrowed from the ecology literature, see for example \cite{Maciel2013,Lam2014,Berestycki2004} and references within.  
A {\it patchy environment} refers to a fragmented environment where each fragment is
relatively homogeneous.  Specifically, we choose to work in an
environment where the heterogeneities are due to variations in the restriction of information of each region.  Thus, $\alpha(x)$ changes by country or region: we can think of 
$\alpha(x)$ as being small for less suppressive states and $\alpha(x)$ as being large for states with high restriction of information.  Through this simple
environment we are able to see a possible cause for why riots (or even revolutionary waves) skip certain countries.  Of course, we cannot conclude that this is the absolute cause;
however, it is interesting to note that a simple set of assumptions will lead to rioting activities failing in certain regions while it propagates globally. 


\subsection{Periodically fragmented environments}  
Here, we  
are interested in understanding when a patchy environment leads to the {\it persistence} of a protest or riot 
and when it leads to these protests vanishing.  For this purpose we study the following problem:
\begin{subequations}\label{times}
\begin{align}
&u_t = \Delta u + \Phi(x,u,v), \\
&v_t = D\Delta v +\Psi(u,v),
\end{align}
\end{subequations}
where $\Phi(x,u,v)$ is periodic in the spatial variable.  For the sake of generality let us assume that $x\in \Real^d$ for $d\ge 1$ and
let $L_1,L_2,...,L_d>0$ be real numbers.  A periodic function $g:\Real^d\rightarrow \Real$ is a function such that $g(x_1,...,x_k+L_k,...,x_d) = g(x_1,...,x_d)$ 
for any $k\in \set{1,...,d}$ and the period $L$ is defined by:
\[L = (0,L_1)\times(0,L_2)\times\cdots\times (0,L_d).\]
We will assume that $\Phi(x,u,v)$ and $\Psi(x,u,v)$ are periodic in $x,$ a natural example being
the case when $\Phi(x,u,v) = r(v)G_{\alpha(x)}(u) -1$ where: 
\begin{align}
\alpha(x) =\left\{ \begin{array}{ll}
\alpha_1 & x \in S_1,\\
\alpha_2 & x\in L\backslash S_1,\end{array}\right.
\end{align}
where $\alpha_1>\alpha_2$ and $S_1\subset L$.  Following the notion of stability of the non-excited state in \cite{Berestycki2004}, 
we look for the principal eigenvalue, $\lambda$, of the operator ${\bf \mathcal{L}}=(\mathcal{L}_1, \mathcal{L}_2)$ defined by:
\begin{align*}
&\mathcal{L}_1:= -\Delta u - \Phi_u(x,0,0) u - \Phi_v(x,0,0) v,\\
&\mathcal{L}_2:= -D\Delta v - \Psi_u(0,0) u - \Psi_v(0,0) v,
\end{align*}
with periodic conditions. Thus, we look for $\lambda$ such that 
there exists positive functions $\phi,\psi$ that satisfy:

\begin{align} \label{sys:pev}
\left\{
\begin{array}{l}
 -\Delta  \phi - \Phi_u(x,0,0) \phi - \Phi_v(x,0,0) \psi= \lambda \phi, \\
-\Delta \psi - \Psi_u(0,0) \phi - \Psi_v(0,0) \psi= \lambda \psi,\\
\phi,\psi>0 \quad \text{periodic}, \quad \norm{\phi}_\infty =\norm{\psi}_\infty=1. 
\end{array}\right.
\end{align}
It can be shown that there exists a unique principal eigenvalue, which is simple, that satisfies \eqref{sys:pev}.  Indeed, this follows by
noticing that replacing $\lambda$ by $\lambda+k$ leads to a cooperative system.  Then the strong maximum principle applied and an application of
the Krein-Rutman theorem gives the result. 
In the case when $\lambda<0$ we say that the non-excited state $(0,v^*(0))$ is unstable and thus we expect the existence of a positive
steady-state solution.  This is stated rigorously in the following proposition.  

\begin{prop}\label{prop:unst}
If $\lambda<0$ there exists a positive and periodic steady-state solution to \eqref{times}.  
\end{prop}
 
\begin{proof}
Since $\lambda<0$ we can choose $k$ sufficiently small so that 
\[
\Phi(k\psi,k\phi) - \frac{\lambda k}{2}\phi > k\psi\Phi_v(0,0)+k\phi\Phi_u(0,0),
\]
and similarly,
\[
\Psi(k\psi,k\phi) - \frac{\lambda k}{2}\psi > k\psi\Psi_v(0,0)+k\psi\Phi_u(0,0),
\]
Then multiplying the top equation of \eqref{sys:pev} by $k$ gives: 
\begin{align}\label{ineq:sub}
&-k\Delta  \phi - \Phi(k\psi,k\phi) - \frac{\lambda}{2} \phi \le 0,	\\
&-k\Delta  \phi - \Psi(k\psi,k\phi) - \frac{\lambda}{2} \phi \le 0.
\end{align}
The inequalities \eqref{ineq:sub} tells us that for $k$ sufficiently small $(k\psi(x),k\phi(x))$
is a sub-solution.  Moreover, $(u^*,v^*(u^*))$ is a super-solution.  Finally, as $(u^*,v^*(u^*))>(k\psi(x),k\psi)$
then, since the system is cooperative, a classical iteration scheme yields a periodic solution $(p(x),q(x))$ with the property that 
$k\phi(x)\leq p(x)\le u^*$ and $k\psi(x)\le q(x)\le v(u^*).$
\end{proof}

As noted earlier the notion of traveling wave solutions is important because these solutions represent 
the spread of high levels of rioting activity and tension into regions with low levels of activity.   
In heterogeneous environments this notion, as introduced in the section \ref{sec:tw}, 
does not apply.  However, the need for such type of solutions still remains true which has lead to the generalized notion of such 
solutions in periodic domains.  These solutions are refereed to as {\it pulsating fronts} and they propagate in a given direction 
but with a profile that changes periodically \cite{Berestycki2002}.  
In other words, these fronts
connect the non-excited state and a periodic steady-state.  We observe such solutions in numerical simulations
and illustrate them in Figure \ref{fig:puls}.  In particular, when there is no restriction of information we observe 
pulsating fronts whose speed depends on the transition parameter.  Figure \ref{fig:puls1} illustrates a fast moving
wave that corresponds to a slow transition rate.  On the other hand, Figure \ref{fig:puls2} illustrates a slower moving
pulsating wave that corresponds to a fast transition rate.  Finally, Figure \ref{fig:puls3} illustrates a lack of 
such type of solutions when there is some restriction of information.  
In fact, Proposition \ref{prop:unst} provides a hint as to when
we expect these pulsating waves to exists.  Specifically, when the principle eigenvalue satisfying \eqref{sys:pev} 
is negative there is hope of a pulsating front existing.  This is an important open problem which we hope to address in the future.

\begin{figure}[H] 
  \center
  \subfloat[$\alpha =0$ and $\beta$ small: $u_0(0) = 5e^{-x}$ ]{\label{fig:puls1}\includegraphics[width=1\textwidth]{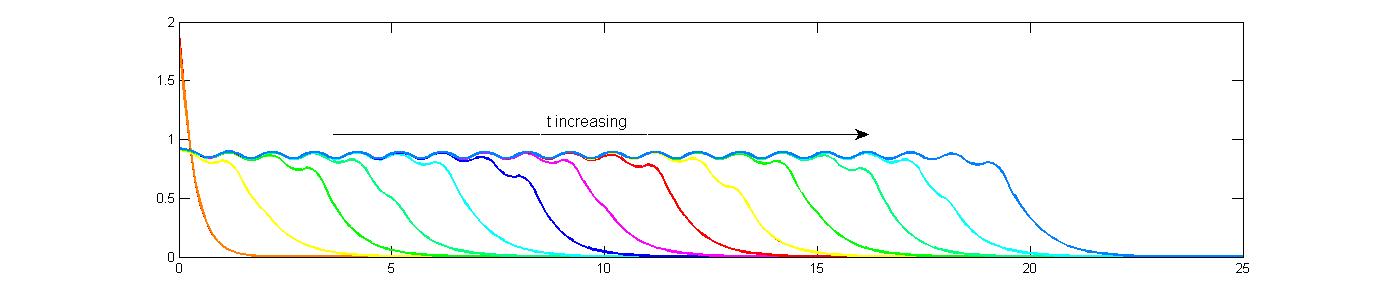}}\\
  \subfloat[$\alpha =0$ and $\beta$ large]{\label{fig:puls2}\includegraphics[width=1\textwidth]{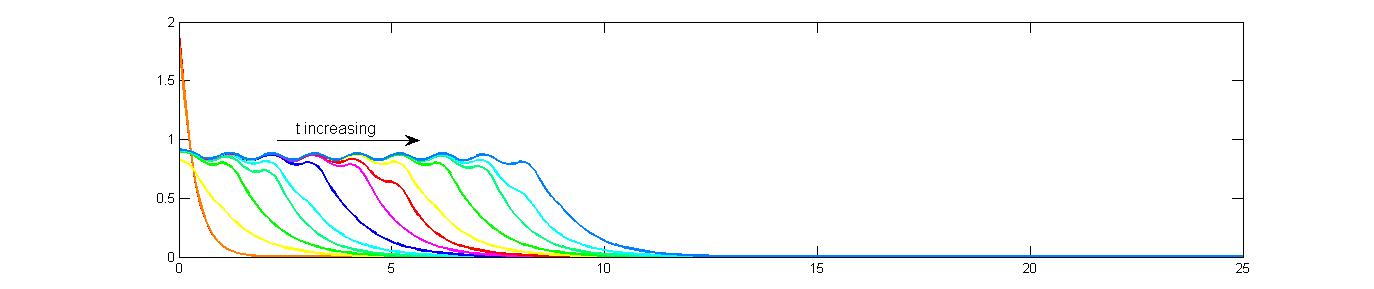}}\\
  \subfloat[$0<\alpha <1$]{\label{fig:puls3}\includegraphics[width=1\textwidth]{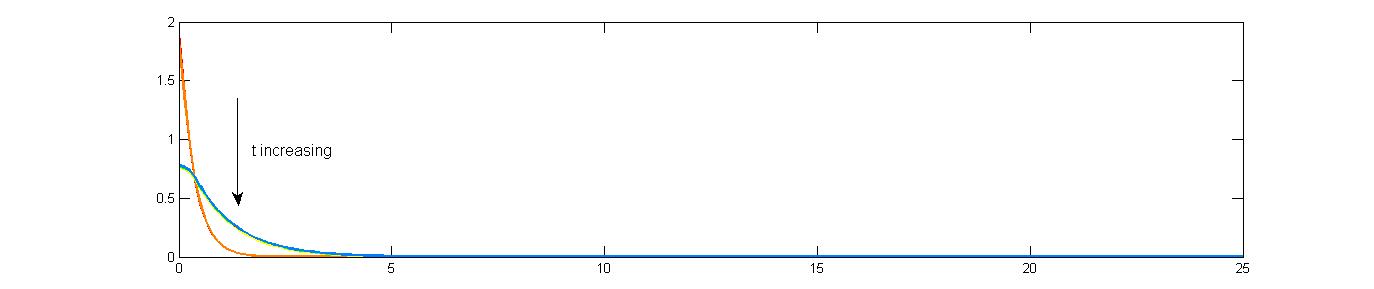}}
 \caption{\small{Pulsating waves in a periodic environment for the level of rioting activity.}\label{fig:puls}}
\end{figure}

\subsection{Gaps in upheaval waves}  
Another typical heterogeneous environment studied in various fields, such as ecology, neural networks, mathematical criminology, and excitable media
(see \cite{Lewis2011,BRR12,Rodriguez2014} and references within) is the so-called gap (or barrier) problem.  In this context, we consider an environment
where there are three regions each with a different level of restriction of information.  
For the sake of illustration let us imagine that we have three neighboring countries, each state is denoted by $S_i,$ for $i = 1,2,3$ in $I\subset\Real.$  
In all of the simulations that we perform the domain of interest is the interval $I=[0,15]$ and the middle country, $S_2,$ has the highest level of
restriction of information, that is $\alpha = 1$.  To make this precise, we define the heterogeneous restriction of information as follows:
\begin{align}
\alpha(x) =\left\{ \begin{array}{ll}
\alpha_1 & x \in S_1,\\
1.0 & x \in S_2,\\
\alpha_2 & x\in S_3,\end{array}\right.
\end{align}
for $0\le \alpha_1,\alpha_2<1.$  Figures \ref{fig:gap1} and \ref{fig:gap2} illustrate bistable and monostable upheaval waves,
respectively, 
which successfully cross the region with high restriction of information.  
\begin{remark}
This offers a possibility as to why
protests or riots fail in certain countries, while being globally successful.    If the country has a high restriction of information it will 
not experience protests even if the neighboring countries have full force protests.   
\end{remark}

In Figures \ref{fig:gap3} and
\ref{fig:gap4} we explore the effect of the size of the middle country on whether the 
revolution is able to cross the country with high restriction of information and take off in the neighboring country.  
In turns out that if the country is sufficiently small
the upheaval wave will successfully cross the country with high restriction of information (still skipping that country), 
this is illustrated in Figure \ref{fig:gap4}.
On the other hand, if the country is too large then this is enough to prevent the upheaval wave from going beyond 
the originating country, this is illustrated in Figure \ref{fig:gap4}.

\begin{figure}[H] 
  \center
  \subfloat[Bistable invasive wave ]{\label{fig:gap1}\includegraphics[width=0.4\textwidth]{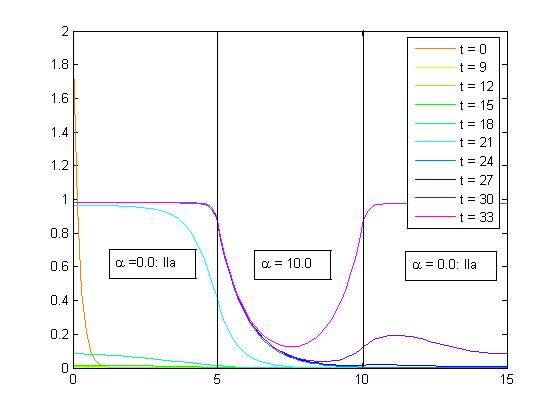}}\quad\quad\quad
 \subfloat[Monostable invasive wave]{\label{fig:gap2}\includegraphics[width=0.4\textwidth]{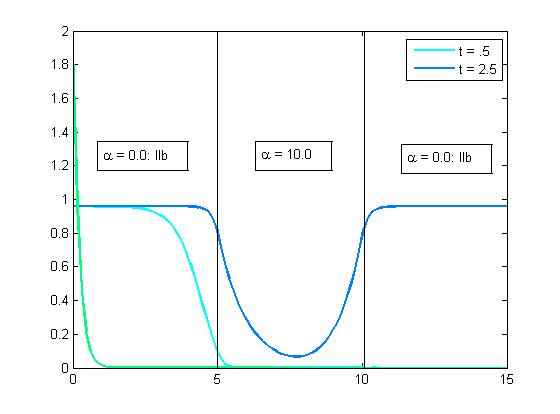}}\\
  \subfloat[Bistable blocked wave: large gap]{\label{fig:gap3}\includegraphics[width=0.4\textwidth]{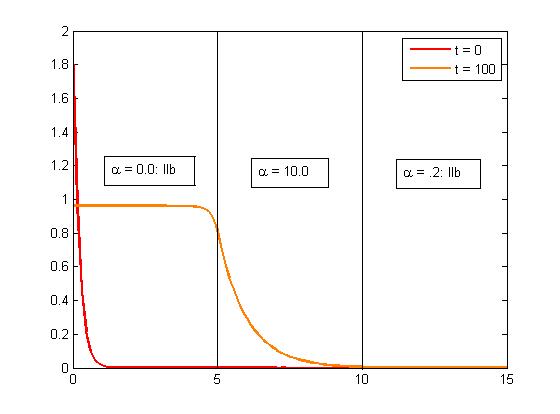}}\quad\quad
 \subfloat[Bistable invasive wave: small gap]{\label{fig:gap4}\includegraphics[width=0.4\textwidth]{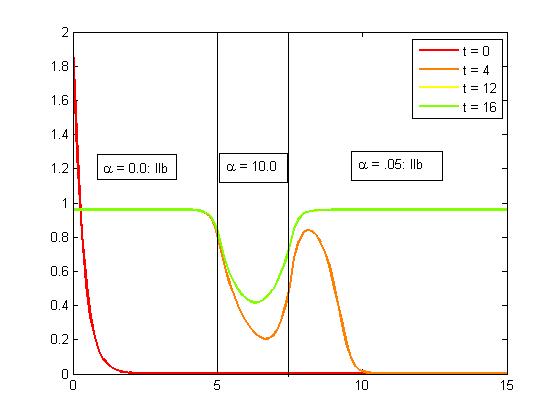}}\\
 \caption{\small{Level of rioting activity in a heterogeneous environment with a gap.}\label{fig:func}}
\end{figure}

\section{Discussion}\label{sec:disc}
This work is a sequel to \cite{Berestycki2014} where we introduce mathematical models for the 
dynamics of rioting activity, which included exogenous and endogenous effects as well as 
non-geographical connections.  
Here we have presented a way to model restriction of information in a given environment, which plays
a significant role on whether rioting activity or protests will be successful, in the sense that they will spread
and last for prolonged periods of time.  This model allowed us to 
explore the effects of different heterogeneities in the environment in the persistence/extinction
of protests and in the existence or non/existence of {\it upheaval waves}.  

We conclude that upheaval waves can, in fact, exist even when there is some restriction of information.
However, when there is no restriction of information these waves can travel arbitrarily fast, which is not the case when 
there is some restriction of information.  As the restriction of information increases we also observe the existence of 
retrieving traveling waves, which essentially are solutions with a negative speed.  From the point of view of the application we have only begun
to scratch the surface and there are many questions that remain to be explored.  For example, the question of whether there are external factors that can help 
stop certain riots is of interest.   

In addition to vast amount of work that remains to be done from the point of view of the application, 
there is also much mathematical theory that remains to be developed for the systems that we introduce here and in \cite{Berestycki2014}.
For example, an extremely interesting open problem is
to prove the existence of traveling wave solutions in the case when the social tension diffuses non-locally.  Such results hold
for single parabolic equitations, under some conditions of the kernel; however, to our knowledge the system has not be treated. 
The existence of pulsating waves in periodic domains, which we observe numerically in our simulations, also remains to be proved.  A Liouville 
type result for the system of equations is also an interesting avenue of research.

\section*{Acknowledgments}
The research leading to these results has received funding from the European Research Council
under the European Union's Seventh Framework Programme (FP/2007-2013) / ERC Grant
Agreement n. 321186 - ReaDi - ``Reaction-Diffusion Equations, Propagation and Modelling'' held by Henri Berestycki. 
This work was also partially supported by the French National Research Agency (ANR), within  project NONLOCAL ANR-14-CE25-0013, and by the French National Center for Scientific Research (CNRS), 
within the project PAIX of the CNRS program PEPS HuMaIn.

The authors are much grateful to Jean-Pierre Nadal for the many interesting and useful discussions. 
	
\section{Appendix}\label{sec:apx}
The appendix is devoted to working through some of the analysis and heuristics used to develop Figures \ref{fig:bif1} and \ref{fig:bif2}.
As noted in section \ref{sub:ss} any constant steady-state solution of \eqref{sys:non_dim_1} with $k = 0$ must satisfy:
\begin{align*}
v^*(u) = \frac{1}{h(u)-k_2}\quad\text{and}\quad \Phi(v^*(u),u)=0.  
\end{align*}
Note that $v^*(u)$ is an increasing function of $u$ such that
\[
\lim_{u\rightarrow \bar u} v^*(u) = \infty, 
\]
for $\bar u = \left(\frac{k}{k_2}\right)^{1/p} -1.$  For simplicity, let us define
\[
H(u):=\Phi(v^*(u),u).  
\]
\subsubsection*{I. $\alpha=0$ case:}
For notational simplicity let us set $G(u)=G_0(u)$. First, we observe that 
$H(u)\le \rho u(1-u)-u$ which is strictly negative for all $u>0$ if $\rho\le 1$.  Now, consider the derivative of $H(u)$ at $u=0$ 
\begin{align}\label{eq:der}
H'(0)=\frac{\rho}{1+e^{-\beta(v^*(0)-\bar a)}}-1,
\end{align}
In particular, for 
\[
\beta\ge \frac{\ln(\rho-1)}{\bar a-v^*(0)}.
\]
we obtain that $H'(0)>0$. This gives the first curve in Figure \ref{fig:bif1}, which is defined by $\beta_1(\rho)=\frac{\ln(\rho-1)}{\bar a-v^*(0)}$ and is the curve where
$H'(0)=0$.  Now, in the region where $H'(0)<0$ we need to determine if $H(u)\ge 0$ for some 
$0<u<\min\set{1,\bar u}$.  For this purpose it is useful to look at the details of $H(u)$
\[
H(u) = \left(\frac{\rho}{1+e^{-\beta (v^*(u)-\bar a)}}-1\right)u-\frac{\rho}{1+e^{-\beta (v^*(u)-\bar a)}}u^2.
\]
\begin{figure}[H] 
\center {\includegraphics[width=0.4\textwidth]{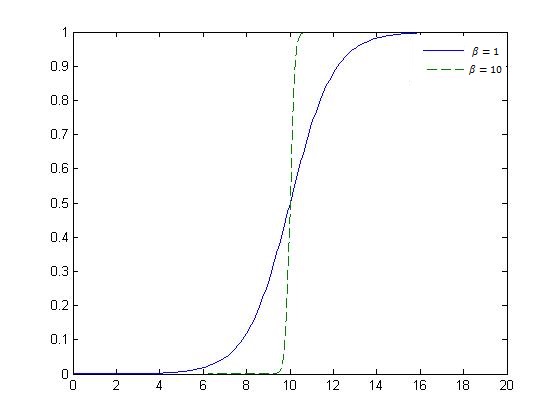}}
  \caption{\small{Illustration of the function $r(z)$ with critical tension $\bar a = 10$ and transition parameter $\beta=1$ and $\beta = 10.$ }}
\end{figure}
Note that for $v^*(u)>\bar a$ for $\beta$ sufficiently large we have that $\frac{\rho}{1+e^{-\beta (v^*(u)-\bar a)}}\approx \rho.$  Thus, for
each $\rho >2$ there exists a $\beta_\rho$ such that $H(u)>0$ for some $u<\min{1,\bar u}$ with $\beta >\beta_\rho$.  In particular $\beta_\rho$ is decreasing.  Moreover, 
if $\rho$ is sufficiently large then $H(u)>0$ for some $u<\min\set{1,\bar u}$ and all $\beta.$  This gives the curve $\beta_2(\rho)$ in Figure \ref{fig:bif1}.  Similarly, we can see that
when in the regime where $H'(0)<0$ but $H(u)>0$ for some $u<\min\set{1,\bar u}$ so that there are three constant steady-state solutions for $\rho$ in intermediate regions
then the quantity $F(u^*)<0$ as defined in \eqref{def:sign} where $u^*$ is the largest zero.  On the other hand, as $\rho$ increases $F(u^*)$ will become positive.  This gives rise to
the curve given by $\beta_4(\rho)$.  Finally, to see where $\beta_3(\rho)$ comes from, first notice that from \eqref{eq:der} that $H'(0)$ increases with $\rho$ from zero.  On top of that,
$H'(0)$ will decrease with $\beta.$ 

\subsubsection*{II. $0<\alpha<1$ case:}

As soon as the restriction of information parameter is positive we see immediately that $H'(0)<0.$  In this case there will generally only be one or three constant steady-state solutions. The possibility
of two steady-state solutions exists on the curve $\beta_1(\rho)$ in Figure \ref{fig:bif2}.  The parameters lying on the curve $\beta_1(\rho)$ leads to a non-zero steady-state solution which is 
degenerate.  In fact, $H(u)<0$ for all $u>0$.  For the curve $\beta_2(\rho)$ we can argue as was done for the $\alpha = 0$ case.

\bibliographystyle{plain}
\bibliography{library}

\end{document}